\documentclass[11pt]{article}
\usepackage[margin=1in]{geometry}
\usepackage{amsmath,amssymb,amsthm,mathtools}
\usepackage{enumitem}
\usepackage{hyperref}
\usepackage{microtype}
\hypersetup{colorlinks=true,linkcolor=blue,citecolor=blue,urlcolor=blue}

\newtheorem{theorem}{Theorem}[section]
\newtheorem{lemma}[theorem]{Lemma}

\newtheorem{corollary}[theorem]{Corollary}
\newtheorem{claim}[theorem]{Claim}
\theoremstyle{definition}
\newtheorem{definition}[theorem]{Definition}
\theoremstyle{remark}

\newcommand{\ceil}[1]{\left\lceil #1\right\rceil}

\title{Optimal binding function for (cap,even hole)-free graphs with no short odd holes}
\author{Chenglong Deng\thanks{School of Mathematical Sciences, Zhejiang Normal University,
Jinhua, Zhejiang 321004, China} \and  Xuding Zhu\thanks{School of Mathematical Sciences, Zhejiang Normal University,
Jinhua, Zhejiang 321004, China}}
\date{}

\begin{document}
\maketitle
\begin{abstract}
   A hole in a graph is an induced cycle of length at least $4$.  A cap is a hole together with a vertex adjacent to exactly two consecutive vertices of it. Chen, Xu and Xu conjectured that if $q\ge2$ and $G$ is a $(\mathrm{cap},\mathrm{even\ hole})$-free graph  with no odd hole of length at most $2q-1$, then $\chi(G)\le \left\lceil \frac{2q+1}{2q}\omega(G)\right\rceil.$
    They confirmed the conjecture for $q \le 3$.  In this paper, we prove the conjecture for all $q \ge 3$. As a corollary, we prove that for  such a graph $G$, $\chi_f(G)\le \frac{2q+1}{2q}\omega(G).$  

\end{abstract}

\section{Introduction}
 For a graph $G$, its clique number $\omega(G)$ is an obvious lower bound for its chromatic number $\chi(G)$. If $\chi(H)=\omega(H)$ for every induced subgraph $H$ of $G$, then $G$ is called a {\em perfect graph}. The famous strong perfect graph conjecture, proposed by Berge \cite{Berge} and confirmed by Chudnovsky, Robertson, Seymour and Thomas \cite{CRST}, gives a characterization of perfect graphs: A graph is perfect if and only if it has no odd hole and no odd anti-hole, where a {\em hole} is an induced cycle of length at least four, and an {\em anti-hole}  is an induced subgraph isomorphic to the complement of a cycle of length at least 5.  A hole or an anti-hole is even or odd if it has an even or an odd number of vertices. 

A classical result of Erd\H{o}s \cite{Erdos1959} says that there are graphs of arbitrarily large girth and arbitrarily large chromatic number. In particular, there are graphs with $\omega(G)=2$ and with arbitrary large chromatic number. 

In 1975, Gy\'{a}rf\'{a}s introduced the concept of $\chi$-bounded family of graphs: 
A class $\mathcal G$ of graphs is called {\em $\chi$-bounded}  if there is a function $f: \mathbb{N} \to \mathbb{N}$ such that $\chi(G)\le f(\omega(G))$ for every $G\in\mathcal G$. Since then, the problem of determining which graph families  are $\chi$-bounded  has been studied extensively in the literature (see \cite{ScottSeymour2020} for a survey). 

For two graphs $G$ and $H$, we say $G$  is {\em $H$-free} if $G$ does not contain an induced subgraph isomomorphic to $H$. If $\mathcal H$ is a family of graphs, then $G$ is $\mathcal H$-free if $G$ is $H$-free for every $H\in\mathcal H$.    The main problem in this area is for which graph family $\mathcal H$, the family of $\mathcal H$-free graphs is $\chi$-bounded. 

It follows from Erd\H{o}s' result that  the family of   $\mathcal H$-free graphs is not $\chi$-bounded, unless either 
$\mathcal H$ contains a forest or the lengths of holes in graphs in  $\mathcal H$  are unbounded.  
Gy\'{a}rf\'{a}s \cite{Gyarfas1975} and Sumner \cite{Sumner1981} conjectured independently that if $\mathcal H$ does contain a forest, then $\mathcal H$-free graphs is $\chi$-bounded (equivalently, for any tree $T$, the family of $T$-free graphs is $\chi$-bounded). This conjecture received a lot of attention \cite{Gyarfas1987, ScottSeymour2020} and is confirmed for some special families of trees, but remains open in general. 

Hole-free graphs are called {\em chordal} graphs, and they are perfect. Gy\'{a}rf\'{a}s \cite{Gyarfas1987} conjectured that the family of odd hole-free graphs, the family of (hole of length at least $\ell$)-free graphs, and the family of (odd hole of length at least $\ell$)-free graphs are $\chi$-bounded. These conjectures remained open for more than 40 years, and are confirmed by Scott and Seymour \cite{ScottSeymour},  Chudnovsky, Scott and Seymour \cite{CSS}, and Chudnovsky, Scott, Seymour and Spirkl \cite{CSSS}, respectively. 

Given a $\chi$-bounded family $\mathcal G$ of graphs,
a natural problem is to find an optimal binding function.  It is very rare that optimal (or near optimal)  binding functions for graph families  are known. For most families of graphs that are known to be  $\chi$-bounded, the proven binding functions are enormous, usually   exponential tower functions. For example, the bounding function for odd hole-free graphs $G$ proved by Scott and Seymour is $\chi(G) \le \frac{2^{2^{\omega(G)+1}}}{48(\omega(G)+1)}$. These functions are believed to be far from optimal. Esperet \cite{Esperet2017} asked whether every $\chi$-bounded hereditary family of graphs is poly-$\chi$-bounded, i.e. the binding function is a polynomial. This question was answered in negative by Bria\'{n}ski,   Davies and Walczak \cite{BDW2024}. However, for many known $\chi$-bounded families of graphs, it is still widely believed that the binding function should be polynomials (see \cite{CSSS2023}). 

The family of even hole-free graphs has a simpler structure. 
Chudnovsky and Seymour \cite{ChudnovskySeymour} proved that every even-hole-free graph $G$ has a bisimplicial vertex, i.e.,  a vertex whose neighbors are covered by two cliques. 
This implies that $\chi(G)\le 2\omega(G)-1$. It is unknown if this binding function is tight.

A {\em cap} is a graph consisting of a hole and a vertex adjacent to exactly two consecutive vertices of the cycle. The structure of $(\mathrm{cap},\mathrm{even\ hole})$-free graphs is even simpler. 
Cameron, da Silva, Huang and Vuskovic proved that $(\mathrm{cap},\mathrm{even\ hole})$-free graphs $G$ have $\chi(G)\le 3\omega(G)/2$ \cite{Cameron2018}. This upper bound was improved by Wu and Xu, and  Xu in \cite{WuXu,Xu2021}, and 
very recently, Chen, Xu and Xu proved the optimal   binding function for $(\mathrm{cap},\mathrm{even\ hole})$-free graphs  \cite{ChenXuXu}: every $(\mathrm{cap},\mathrm{even\ hole})$-free graph $G$ satisfies $\chi(G)\le \left\lceil {5\over 4}\omega(G)\right\rceil$. This bound is attained by the lexicographic product of $C_5$ and $K_k$: $\chi(C_5[K_k])=\lceil 5k/2 \rceil =\lceil 5\omega(G)/4 \rceil$. 
Chen, Xu and Xu further proved that every $(\mathrm{cap},\mathrm{even\ hole},5\text{-hole})$-free graph $G$ satisfies $  \chi(G)\le \left\lceil {7\over 6}\omega(G)\right\rceil$. 
Chen, Xu and Xu conjectured the following general pattern: if $q\ge 2$ and $G$ is a $(\mathrm{cap},\mathrm{even\ hole})$-free graph with no odd hole of length at most $2q-1$, then 
$ 
        \chi(G)\le \left\lceil {2q+1\over 2q}\omega(G)\right\rceil$.
This bound is   sharp, as for $G = C_{2q+1}[K_k]$, $\chi(G)=\lceil \frac{2q+1}{q} k \rceil = \lceil \frac{2q+1}{2q}\omega(G) \rceil$.

Their result confirms this conjecture for $q \le 3$. In this paper, we confirm this conjecture for all $q \ge 3$ (and hence the conjecture is completely confirmed). 

\begin{theorem}\label{thm:main}
 Assume $G$ is a $(\mathrm{cap},\mathrm{even\ hole})$-free graph. If $G$ has no odd hole of length at most $2q-1$ for some integer $q \ge 3$, then
\[
        \chi(G)\le \ceil{\frac{2q+1}{2q}\omega(G)}.
\]
\end{theorem}

As a corollary, we prove the following upper bound on the fractional chromatic number of these graphs. 

\begin{corollary}\label{cor:fractional}
  If $q \ge 2$ and $G$ is a $(\mathrm{cap},\mathrm{even\ hole})$-free graph with no odd hole of length at most $2q-1$, then
\[
        \chi_f(G)\le \frac{2q+1}{2q}\omega(G).
\]
\end{corollary}
\begin{proof}
     By Theorem \ref{thm:main} and the result in \cite{ChenXuXu}, 
    $\chi(G[K_{2q}]) \le \ceil{\frac{2q+1}{2q}\omega(G[K_{2q}])} =  (2q+1) \omega(G)$. Therefore 
    $\chi_f(G) \le \frac{\chi(G[K_{2q}])}{2q} \le \frac{2q+1}{2q}\omega(G)$.
\end{proof}

\section{Proof of the theorem}

A {\em clique cutset} in a graph $G$ is a clique $K$ such that $G-K$ has more connected components than $G$.  A clique $K$ is {\em  universal}  if all vertices in $K$ are universal, i.e., adjacent to every other vertex.  

Assume $G$ is a graph and   $x,y,z$ are three consecutive vertices of a hole in $G$.  By {\em adding an ear} to $G$ with $x,z$ as its attachment, we mean adding a $x$-$z$-path $P$  whose internal vertices are disjoint from $V(G)$ and adding edges joining $y$ to some vertices of $P$.   A wheel $W$ is a graph obtained from a cycle $C$ by adding a vertex $v$ adjacent to at least three vertices of $C$. The added vertex is called the {\em center} of $W$, and the cycle $C$ is called the {\em ring} of $W$.  The ear addition is called {\em good}  if the following   hold: 

\begin{itemize}
    \item $y$ has an odd number of neighbors on $V(P)$.
    \item There is no wheel $W$ in $G$ with $v$ as its center, $x,y,z $ on the ring and $vy\in E(G)$.
    \item There is no wheel $W$ in $G$ with $y$ as its center, and  both $x$ and $z$ are neighbors of $y$ in $W$.
\end{itemize}

Let $G$ be a graph, and let $f$ be an assignment of integers 0 and 1 to its edges. A subgraph $H$
of $G$ is said to be odd (resp. even) if $\sum_{e \in E(H)} f(e)$ is odd (resp. even). The graph $G$ is said to
be odd-signable if it has  an assignment $f: E(G) \to \{0,1\}$ such that every induced cycle is odd.
Note that
every even hole-free graph $G$ is odd-signable by assigning each edge of $G$ with integer 1.
We shall use the following two  lemmas.

\begin{lemma}\label{lem:good-ear}\cite{Conforti2000}
Let $F$ be a triangle-free graph with at least three vertices.  Suppose $F$ is not the cube and does not have a clique cutset.  Then $F$ is odd-signable if and only if it can be obtained, starting from a hole, by a sequence of good ear additions.
\end{lemma}

\begin{lemma}\label{lem:structure}\cite{Cameron2018}
Let $G$ be a $(\mathrm{cap},4\text{-hole})$-free graph that contains a hole and does not have a clique cutset.  Let $F$ be a maximal triangle-free induced subgraph of $G$ with at least three vertices and without clique cutset.  Then $G$ is obtained from $F$ by first blowing up vertices of $F$ into nonempty cliques and then adding a universal clique.
\end{lemma}

\noindent
{\bf Proof of   Theorem~\ref{thm:main}}.  Suppose that the theorem is false, and $G$ is a counterexample  with $|V(G)|$ minimum.  Put
\[
        K=\ceil{\frac{2q+1}{2q}\omega(G)}.
\]
Then $\chi(G)>K$.

By the minimality of $G$, the graph $G$ is connected and is not perfect.  Moreover, $G$ has no clique cutset $Z$.  Indeed, if $G$ has a clique cutset $C$ and $G_1, \ldots, G_k$ are the components of $G-Z$, then the subgraphs $G'_i=G_i \cup Z$ of $G$  can be colored with $K$ colors by the minimality of $G$. The  colorings can be made to agree on the clique $C$ by permuting colors on the subgraphs $G'_i$, and hence their union is a $K$-coloring of $G$, a contradiction.  

The graph $G$ also has no non-empty universal clique.  If $U$ were such a clique, then $\chi(G)=|U|+\chi(G-U)$ and $\omega(G)=|U|+\omega(G-U)$; the minimality of $G-U$ would give
\[
        \chi(G)\le |U|+\ceil{\frac{2q+1}{2q}\omega(G-U)}
        \le \ceil{\frac{2q+1}{2q}\omega(G)},
\]
a contradiction.

Since $G$ is not perfect and even-hole-free, it contains an odd hole.  Choose a maximal triangle-free induced subgraph $F$ of $G$ with no clique cutset.  By Lemma~\ref{lem:structure},   $G$ is a clique blowup of $F$.
We call $F$ the {\em skeleton} of $G$.   Since $G$ is even-hole-free and has no odd hole of length at most $2q-1$,    $F$ is also even-hole-free  and has no odd hole of length at most $2q-1$.

   Let $w(v)$ be the size of the clique replacing $v\in V(F)$.  Since $F$ is triangle-free, the clique number of $G$ is
\[
        W=\omega_w(F)=\max\{w(u)+w(v):uv\in E(F)\}.
\]

\begin{definition}
    \label{def-col}
    Assume $F$ is a graph, $w$ is a weight assignment that assigns to each vertex $v$ of $F$ a positive integer $w(v)$, and $K$ a positive integer. A $K$-coloring of $(F,w)$ is a mapping $\phi$ that assigns to each vertex $v$ of $F$ a subset $\phi(v)$ of $[K]$  such that $|\phi(v)|=w(v)$ and $\phi(u)\cap \phi(v)=\varnothing$ whenever $uv\in E(F)$.
\end{definition}

To prove Theorem \ref{thm:main},   it suffices to prove the following lemma.

\begin{lemma}
    \label{lem:bounded}
Let $q \ge 3$ and $F$ be a triangle-free even-hole-free graph with no odd hole of length at most $2q-1$ and with no clique cut.  Let $w:V(F)\to \mathbb Z_{>0}$ be a weight function,   $W=\omega_w(F)$ and $ K=\ceil{\frac{2q+1}{2q}W}$. 
Then $(F,w)$ has a $K$-coloring.
\end{lemma}

Assume Lemma~\ref{lem:bounded} is not true and $F$   is a counterexample with minimum number of vertices. 

If $F=(v_0,v_1, \ldots, v_{2r})$ is an odd cycle of length $2r+1$ with $r\ge q$, then by symmetry, we may assume that $w(v_{2r}) \le W/2$. For $i=0,1, \ldots, r-1$, let 
  \begin{eqnarray*}
      \phi(v_{2i}) &=& (iW, iW+w(v_{2i})] \pmod{K}, \\
      \phi(v_{2i+1}) &=& (iW+w(v_{2i}), iW+w(v_{2i})+w(v_{2i+1})] \pmod{K},
  \end{eqnarray*}  
  where $(a, b] \pmod{K} = \{c: c \in \{1, 2,\ldots, K\}, \exists c', a < c' \le b, c' \equiv c \pmod{K}\}$. Let $\phi(v_{2r}) = (rW, rW+w(v_{2r})] \pmod{K}$.
Then $|\phi(v_i)| = w(v_i)$ and  $\phi(v_i) \cap \phi(v_{i+1}) = \emptyset$ for $i=0,1, \ldots, 2r$ (addition in the indices modulo $2r+1$). So $\phi$ is a $K$-coloring of $(F,w)$. 

Assume that  $F$ is not an odd cycle.
Since $F$ is even-hole-free, it is odd-signable. Also $F$ is not the cube. By Lemma~\ref{lem:good-ear},  $F$ is obtained from a smaller even-hole-free graph 
$F'$ by adding a good ear  
\[
        P=x_0x_1\cdots x_s.
\]
Let $y$ be the vertex of an old hole in $F'$ such that $x_0,y,x_s$ are consecutive vertices of the odd hole and   $y$ has an odd number of neighbors in $P$.

Choose three consecutive neighbors $x,a ,z$ of $y$ in $P$ in this order. 
Thus  $y$ has no neighbor in the interior of either segment $xPa$ or $aPz$, and  $y+xPa+y$ and $y+aPz+y$ are induced holes.  Since $F$ has no even hole and no odd hole of length at most $2q-1$, these two holes are odd and have length at least $2q+1$.  Consequently, the two paths $xPa$ and $aPz$ have odd lengths at least $2q-1$.

Let $$xPa = v_1v_2 \ldots v_{2m}, \, zPa=u_1u_2\ldots u_{2n}$$
be the segments of $P$ from $x$ to $a$ and from $z$ to $a$, respectively.
(So $v_1=x, u_1=z$ and $v_{2m}=u_{2n}=a$.) Then $m,n \ge q$. 

Let $B = \{v_2,v_3, \ldots, v_{2m}, u_2, u_3, \ldots, u_{2n-1}\}$, and $F^-=F-B$.   By the minimality of $F$, there is a $K$-coloring $\phi$ of $(F^-, w)$.  

We shall show that $\phi$ can be extended to a $K$-coloring of $(F, w)$.

Let $X=\phi(x), Y=\phi(y), Z= \phi(z)$.  We need to  find a subset $A$ of $[K]-\phi(y)$ of size $w(a)$, so that $(xPa, w)$ has a $[K]$-coloring $\psi$ with $\psi(v_1)=X$, $\psi(v_{2m})=A$, and  $(zPa, w)$ has a $[K]$-coloring $\psi$ with $\psi(u_1)=Z$, $\psi(u_{2n}) = A$. 

The next lemma gives a sufficient condition that the pre-coloring of the two end vertices can be extended to a $[K]$-coloring of the whole path.

\begin{lemma}\label{lem:path}
Let
\[
        P=v_1v_2\cdots v_{2m}
\]
be a path and $w$ is a weight assignment that assigns to each vertex $v_i$ a positive integer $w(v_i)$.
Suppose that for each \(i\), 
\[
        w(v_i)+w(v_{i+1}) \le W,
\]
and \(K\ge W\).   Let 
\[
        S=\sum_{i=2}^{2m-1}w(v_i).
\]
If $A_1$ is a $w(v_1)$-subset of $[K]$, $A_{2m}$ is a $w(v_{2m})$-subset of $[K]$ such that $$  |A_{2m} \cap A_1|\le (m-1)K-S,$$ then there is a \(K\)-coloring $\phi$ of $(P, w)$ such that $\phi(v_1)=A_1$ and $\phi(v_{2m})=A_{2m}$.
\end{lemma}

\begin{proof}
Note that $S = \sum_{j=1}^{m-1} (w(v_{2j} + w(v_{2j+1}) \le (m-1)W \le  (m-1)K$. So $(m-1)K-S \ge 0$.

The proof is by induction on $m$. If $m=1$, then $(m-1)K-S=0$. For \(A_1, A_2 \subseteq [K]\),  \(|A_i|=w(v_i)\) with $A_1 \cap A_2 = \emptyset$, $\phi(v_i)=A_i$ for $i=1,2$ is a \(K\)-coloring of $(P_2, w)$.

Assume $m \ge 2$ and the lemma holds for a path of length $2(m-1)$. Let \(A_1, A_{2m}\) be subsets of \(  [K]\) such that \(|A_1|=w(v_1)\),  \( |A_{2m}|=w(v_{2m})\) and \(|A_1\cap A_{2m}|\le (m-1)K-S\). 

Let $P'=v_1v_2\ldots v_{2m-2}$ and $S'=\sum_{i=2}^{2m-3}w(v_i)$. 

Then $[K]-A_{2m}$ has a subset $A_{2m-1}$ of size $w(v_{2m-1})$ such that 
$$|A_{2m-1} \cap A_1| \ge \min\{w(v_{2m-1}), w(v_1)-((m-1)K-S)\}.$$ This implies that 
$$|A_{2m-1} \cup A_1| =w(v_1)+w(v_{2m-1}) - |A_{2m-1} \cap A_1| \le   \max\{w(v_1), w(v_{2m-1})+((m-1)K-S)\}.$$
Then  $[K]-A_{2m-1}$ has a subset $A_{2m-2}$ of size $w(v_{2m-2})$ with
 \begin{eqnarray*}
     |A_{2m-2} \cap A_1| &\le& \max\{w(v_1)+w(v_{2m-2})-K,  w(v_{2m-2}) + w(v_{2m-1}) + (m-2)K-S \} \\  
     &\le& (m-2)K-S'.
 \end{eqnarray*} 

By the induction hypothesis, there is a $K$-coloring $\phi$ of $(P',w)$  with $\phi(v_1)=A_1$ and $\phi(v_{2m-2})=A_{2m-2}$. Then  $\phi$ can be extended further to a $K$-coloring of $(P,w)$ with $\phi(v_{2m})=A_{2m}$. 

This completes the proof of Lemma \ref{lem:path}.
\end{proof}

 Let
\[
        S_L=\sum_{i=2}^{2m-1}w(v_i), \text{ and } 
        S_R=\sum_{i=2}^{2n-1}w(u_i).
\]

\begin{definition}
    \label{def-avoidable}
    A subset $H$ of $[K]\setminus Y$ is {\em good} if $|H| \le K-w(a)-w(y)$ and 
 $$
        |H\cap X|\ge w(x)+S_L-(m-1)K \text{ and } 
        |H\cap Z|\ge w(z)+S_R-(n-1)K. \eqno{(1)}
        $$
\end{definition}

\begin{lemma}
    \label{lem-suff}
   If $[K]\setminus Y$ has a good subset, 
 then $\phi$ can be extended to a $K$-coloring of $(F, w)$.
\end{lemma}
\begin{proof}
    Let $H$ be a good subset of $[K] \setminus Y$. Then   $[K] - (Y \cup H)$ has a subset $A$ of size $w(a)$. It follows from ~(1) that  
 $$ | A \cap X|\le (m-1)K-S_L \text{ and } 
 | A \cap Z|\le (n-1)K-S_R.$$
 First, we extend $\phi$ to $a$ by letting $\phi(a)=A$. By Lemma \ref{lem:path}, $\phi$ can be extended further to a $K$-coloring of $(F, w)$.
\end{proof}

It remains to prove that $[K]\setminus Y$ has  a good subset.

Let $$r_L =   w(x)+S_L-(m-1)K,  \text{ and }  r_R = w(z)+S_R-(n-1)K.$$
Let $r_L^+ \max\{0, r_L\}$ and $r_R^+ = \max\{0, r_R\}$.
\begin{lemma}\label{lem:select}
If $K-w(y)-w(a) \ge \max\{r_L, r_R, r_L+r_R-|X \cap Z| \}$, then $[K]\setminus Y$ has  a good subset.
\end{lemma}

\begin{proof}
 By symmetry,  we may assume that $r_L\le r_R$.   If $|X \cap Z| \ge r_L$, choose $r_L^+$ elements from $X\cap Z$.  Add $r_R^+-r_L^+$ other elements from $Z$.  This is possible since $S_R \le (n-1)K$ and hence $r_R\le w(z)=|Z|$, and the resulting set $H$ has size $r_R\le K-w(y)-w(a)$, and therefore is a good subset of $[K]-Y$.

If $|X \cap Z| <r_L$, choose all elements of $X\cap Z$,  add $r_L-|X \cap Z|$ elements from $X\setminus Z$ and $r_R-|X \cap Z|$ elements from $Z\setminus X$.  This is possible because $r_L\le |X|$ and $r_R\le |Z|$.  The resulting set $H$ has size 
\[
        |X \cap Z| +(r_L-|X \cap Z|)+(r_R-|X \cap Z|)=r_L+r_R-|X \cap Z|\le K-w(y)-w(a).
\]
Hence $H$ is a good subset of $[K] \setminus Y$.  This proves the lemma.
\end{proof}

It remains to prove that  
$$K-w(y)-w(a) \ge \max\{r_L, r_R, r_L+r_R-|X \cap Z| \},$$

We first prove $K-w(y)-w(a) \ge r_L$. The vertices
$y,v_1,\ldots,v_{2m},y$ form an odd hole $C$ of length $2m+1$, where
$v_1=x$ and $v_{2m}=a$. 
Therefore
$$
2(w(y)+w(x)+S_L+w(a)) = \sum_{uv \in E(C)} (w(u)+w(v)) \le (2m+1)W.
$$
Since $m\ge q$ and $K\ge \frac{2q+1}{2q}W$, we have
$\frac{2m+1}{2}W\le mK$. Hence
$$r_L=w(x)+S_L-(m-1)K\le K-w(y)-w(a).$$ 

The same argument applied to the
odd hole $y,u_1,\ldots,u_{2n},y$ gives $r_R\le K-w(y)-w(a)$.

Next, we prove  $K-w(y)-w(a)  \ge r_L+r_R - |X \cap Z|$.  If $r_L \le 0$, then this follows from the result that $K-w(y)-w(a)  \ge r_R$.   
Similarly, if $r_R \le 0$, then this follows from the result that $K-w(y)-w(a)  \ge r_L$. 
  Assume $r_L, r_R > 0$. 

Let
$$
L=\min\{w(x),W-w(a)\},\qquad
R=\min\{w(z),W-w(a)\}.
$$

\begin{claim}
\label{clm-rlrr}
    $r_L \le L-(q-1)(K-W)$ and $r_R\le R-(q-1)(K-W)$. Consequently, $r_L+r_R \le L +R -2(q-1)(K-W)$.
\end{claim}
 \begin{proof}
     Pairing the
internal vertices of $xPa$ gives $S_L\le(m-1)W$.  Therefore
$(m-1)K-S_L\ge (m-1)(K-W)$. This implies
$$r_L=w(x) +S_L -(m-1)K \le w(x) - (q-1)(K-W). \eqno{(2)}$$

 On the other hand, pairing the whole
segment $xPa$ gives $w(x)+S_L+w(a)\le mW= (m-1)K- (m-1)(K-W) +W$. 
Hence 
$$r_L= w(x)+S_L-(m-1)K\le W-w(a)-(m-1)(K-W) \le W-w(a)-(q-1)(K-W). \eqno{(3)}$$ 
It follows from ~(2) and ~(3) that 
$r_L \le L - (q-1)(K-W)$.

By symmetry, $r_R\le R-(q-1)(K-W)$.
 \end{proof} 

 To complete the proof of Theorem \ref{thm:main}, 
 it suffices to prove that $$L+R-2(q-1)(K-W)\le K-w(y)-w(a) + |X \cap Z|. \eqno{(4)}$$

Since $X,Z\subseteq [K]-Y$, we have
$$ |X\cap Z|\ge  |X|+|Z|+|Y|-K \ge   L+R+|Y|-K. \eqno{(5)} $$  Plugging ~(5) into ~(4),  we conclude that ~(4) is implied by  $$  2(q-1)(K-W)\ge    w(a). \eqno{(9)}$$

It follows from Claim \ref{clm-rlrr} that  $W-w(a) \ge L \ge r_L + (q-1)(K-W) > (q-1)(K-W) $.
As $W \le 2q(K-W)$ and $q \ge 3$,
$$w(a)< W-(q-1)(K-W) \le (q+1)(K-W) \le 2(q-1)(K-W).$$

This completes the proof of Theorem \ref{thm:main}.

\bibliographystyle{plain}
\bibliography{ref.bib}

\end{document}